\documentclass[10pt]{amsart}
\usepackage{amssymb,amscd}
\usepackage[all,cmtip]{xy}
\usepackage{graphics}
\usepackage{epic}
\usepackage{url}
\usepackage[T2A]{fontenc}
\usepackage[english]{babel}

\usepackage[utf8x]{inputenc}
\usepackage{multirow}

\usepackage{algorithm}
\usepackage{algpseudocode}
\usepackage{longtable}
\usepackage{amsmath}
\usepackage{slashbox}

\usepackage{xcolor}

\newtheoremstyle{observation}
  {\topsep}
  {\topsep}
  {}
  {}
  {\bfseries}
  {.}
  {.5em}
  {}
\theoremstyle{observation}

\newtheorem{theorem}{Theorem}[section]
\newtheorem{lemma}[theorem]{Lemma}
\newtheorem{proposition}[theorem]{Proposition}
\newtheorem{claim}[theorem]{Claim}
\newtheorem{corollary}[theorem]{Corollary}

\newtheorem{definition-lemma}[theorem]{Definition-Lemma}
\newtheorem{definition}[theorem]{Definition}

\newtheorem{remark}[theorem]{Remark}

\def\CC{{\mathbb C}}
\def\DD{{\mathbb D}}

\def\FF{{\mathbb F}}

\def\QQ{{\mathbb Q}}

\def\ZZ{{\mathbb Z}}
\def\cO{{\mathcal O}}

\newcommand\Hom{\operatorname{Hom}}

\newcommand\Spec{\operatorname{Spec}}

\begin{document}

\title[Computing algorithm for reduction type of CM Abelian varieties]{Computing algorithm for reduction type of CM Abelian varieties}

\author{Artyom Smirnov and Alexey Zaytsev}

\email{asmirnov1005@gmail.com, al.zaytsev@skoltech.ru, azaytsev@sk.ru}
\address{Skolkovo Institute of Science and Technology,
Nobelya Ulitsa 3.
Moscow, Russia} 

\maketitle

\begin{abstract}
	Let $\mathcal{A}$ be an Abelian variety over a number field, with a good reduction at a prime ideal containing a prime number $p$. Denote by ${\rm A}$ an Abelian variety over a finite field of characteristic $p$, obtained by the reduction of $\mathcal{A}$ at the prime ideal. In this paper we derive an algorithm which allows to decompose the group scheme ${\rm A}[p]$ into indecomposable quasi-polarized ${\rm BT}_1$-group schemes. This can be done for the unramified $p$ on the basis of its decomposition into prime ideals in the endomorphism algebra of ${\rm A}$. We also compute all types of such correspondence for Abelian varieties of dimension up to $5$.  As a consequence we establish the relation between the decompositions of prime $p$ and the corresponding pairs of $p$-rank and $a$-number of an Abelian variety ${\rm A}$.	
\end{abstract}

\begin{section}{Introduction and Results}\label{Introduction}
	
	Let $\mathcal{A}$ be a simple Abelian variety of dimension $g$ that has a complex multiplication by the full ring of integers $\cO_K$ of CM field $K$. Assume that $\mathcal{A}$ has a good reduction at prime $\mathcal{P}$ of $\cO_K$, denoted by $\rm A$. In addition, we suppose that prime $p = \mathcal{P} \cap \ZZ$ is unramified in $K$.

	In this paper we produce a general algorithm for the calculation of the correspondences between the decompositions of $p$ into the prime ideals in $\cO_K$ and the decompositions of the group scheme ${\rm A}[p]$ into the quasi-polarized indecomposable ${\rm BT}_1$-group schemes. The question of which non-isogenic decompositions of scheme ${\rm A}[p]$ can correspond to one or another decomposition of $p$ has already been considered earlier in \cite{Blake} and \cite{Sugiyama}. But we are considering the admissible decompositions of ${\rm A}[p]$ up to isomorphism. This work extends an approach of \cite{Zaytsev} to higher dimensions. In particular, in the section \ref{Algorithm_Description} we obtain the tables with explicit correspondences for dimensions up to~$5$.
		
	We start with overview of basic results on group schemes and the classification of quasi-polarized indecomposable ${\rm BT}_1$-group schemes of a given dimension (see \ref{Preliminaries_and_Classification}).Then we provide a theoretical justification of the algorithm and we consider a pseudocode along with an example of the algorithm application. The results of the program implementation are presented in the tables at the end of the section \ref{Algorithm_Description} which improve the results in \cite{Zaytsev}.

	Analysis of these results yields to the following connection between the parameters of Abelian variety ${\rm A}$ and the decomposition of $p$ into prime ideals. 
	Let $f$ be $p$-rank and $a$ be $a$-number corresponding to Abelian variety ${\rm A}$. Suppose that $K_0$ is a totally real subfield of $K$ with the full ring of integers $\cO_{K_0}$. For convenience we introduce the notion of \textbf{decomposition type} of prime number $p$ in CM field $K$. Namely, it is a pair of non-negative integers $(\alpha, \beta)$, such that $p \cO_K = \mathcal{P}_1 ... \mathcal{P}_{\alpha}$ and $p \cO_{K_0} = \mathcal{Q}_1 ... \mathcal{Q}_{\beta}$. Then the following claim holds.
	
	\begin{claim}\label{Main_Theorem_for_parameters}
		Let $\mathcal{A}$ be a simple Abelian variety of dimension $g$ over a number field $L$ with complex multiplication by the full ring of integers $\cO_K$ of CM field $K \subset L$. Assume that $\mathcal{A}$ has a good reduction at prime $\mathcal{P}$ of $\cO_K$ and $p = \mathcal{P} \cap \ZZ$ is unramified in $K$. Let $f$ be $p$-rank and $a$ be $a$-number of $\mathcal{A}$ and $(\alpha, \beta)$ be the decomposition type of $p$ in $K$. Then all the possible correspondences between pairs $(a, f)$ and  $(\alpha, \beta)$ arising in such a way can be described by the following table.
	  
  	\scriptsize
  	\begin{center}
  	\begin{tabular}{|c|c|ccccc}
  
  	\hline
  
  	\backslashbox{$f$}{$a$} & $0$ & \multicolumn{1}{|c|}{$1$} & \multicolumn{1}{|c|}{$2$} & \multicolumn{1}{|c|}{$3$} & \multicolumn{1}{|c|}{$4$} & \multicolumn{1}{|c|}{$5$} \\
  
  	\hline
  
  	\multirow{5}{*}{$0$} & \multirow{5}{*}{} & \multicolumn{1}{|c|}{\multirow{5}{*}{$(1, 1)$}} & \multicolumn{1}{|c|}{\multirow{5}{*}{$(2, 1)$, $(2, 2)$}} & \multicolumn{1}{|c|}{} & \multicolumn{1}{|c|}{} & \multicolumn{1}{|c|}{$(1, 1)$} \\
  	& & \multicolumn{1}{|c|}{} & \multicolumn{1}{|c|}{} & \multicolumn{1}{|c|}{$(1, 1)$} & \multicolumn{1}{|c|}{$(2, 1)$, $(2, 2)$} & \multicolumn{1}{|c|}{} \\
  	& & \multicolumn{1}{|c|}{} & \multicolumn{1}{|c|}{} & & \multicolumn{1}{|c|}{} & \multicolumn{1}{|c|}{$(3, 2)$, $(3, 3)$} \\
  	& & \multicolumn{1}{|c|}{} & \multicolumn{1}{|c|}{} & \multicolumn{1}{|c|}{$(3, 2)$, $(3, 3)$} & \multicolumn{1}{|c|}{$(4, 2)$, $(4, 3)$, $(4, 4)$} & \multicolumn{1}{|c|}{} \\
  	& & \multicolumn{1}{|c|}{} & \multicolumn{1}{|c|}{} & \multicolumn{1}{|c|}{} & \multicolumn{1}{|c|}{} & \multicolumn{1}{|c|}{$(5, 3)$, $(5, 4)$, $(5, 5)$} \\
  
  	\hline
  
  	\multirow{2}{*}{$1$} & \multirow{2}{*}{$(2, 1)$} & \multicolumn{1}{|c|}{\multirow{2}{*}{$(3, 2)$}} & \multicolumn{1}{|c|}{\multirow{2}{*}{$(4, 2)$, $(4, 3)$}} & \multicolumn{1}{|c|}{$(3, 2)$} & \multicolumn{1}{|c|}{$(4, 2)$, $(4, 3)$} & \\
  	& & \multicolumn{1}{|c|}{} & \multicolumn{1}{|c|}{} & \multicolumn{1}{|c|}{$(5, 3)$, $(5, 4)$} & \multicolumn{1}{|c|}{$(6, 3)$, $(6, 4)$, $(6, 5)$} & \\ \cline{1-6}
  
  	\multirow{2}{*}{$2$} & $(2, 1)$ & \multicolumn{1}{|c|}{$(3, 2)$} & \multicolumn{1}{|c|}{$(4, 2)$, $(4, 3)$} & \multicolumn{1}{|c|}{$(5, 3)$, $(5, 4)$} & & \\
  	& $(4, 2)$ & \multicolumn{1}{|c|}{$(5, 3)$} & \multicolumn{1}{|c|}{$(6, 3)$, $(6, 4)$} & \multicolumn{1}{|c|}{$(7, 4)$, $(7, 5)$} & & \\ \cline{1-5}
  
  	\multirow{3}{*}{$3$} & $(2, 1)$ & \multicolumn{1}{|c|}{$(3, 2)$} & \multicolumn{1}{|c|}{$(4, 2)$, $(4, 3)$} & & & \\
  	& $(4, 2)$ & \multicolumn{1}{|c|}{$(5, 3)$} & \multicolumn{1}{|c|}{$(6, 3)$, $(6, 4)$} & & & \\
  	& $(6, 3)$ & \multicolumn{1}{|c|}{$(7, 4)$} & \multicolumn{1}{|c|}{$(8, 4)$, $(8, 5)$} & & & \\ \cline{1-4}
  
  	\multirow{4}{*}{$4$} & $(2, 1)$ & \multicolumn{1}{|c|}{$(3, 2)$} & & & & \\
  	& $(4, 2)$ & \multicolumn{1}{|c|}{$(5, 3)$} & & & & \\
  	& $(6, 3)$ & \multicolumn{1}{|c|}{$(7, 4)$} & & & & \\
  	& $(8, 4)$ & \multicolumn{1}{|c|}{$(9, 5)$} & & & & \\ \cline{1-3}
  
  	\multirow{5}{*}{$5$} & $(2, 1)$ & & & & & \\
  	& $(4, 2)$ & & & & & \\
  	& $(6, 3)$ & & & & & \\
  	& $(8, 4)$ & & & & & \\
  	& $(10, 5)$ & & & & & \\ \cline{1-2}
  
  	\end{tabular}
  	\end{center}
  	\normalsize
	\end{claim}

\begin{subsection}*{Acknowledgment}
The first author would like to thank Dr. C. Diem for his hospitality and the valuable remarks during his visit to Leipzig University for the DAAD program. This program was organized thanks to Prof. Dr. M. Droste with the assistance of Dr. Sergey Aleshnikov and Yuri Boltnev. Both authors would like to thank Anastasia Svalova for editing the translation.
\end{subsection}

\end{section}

\begin{section}{Preliminaries and Classification of Quasi-Polarized Indecomposable ${\rm BT}_1$-Group Schemes up to Dimension $5$}\label{Preliminaries_and_Classification}

	In this section we give a brief overview of the category of finite commutative group schemes over a perfect field $k$ with ${\rm char(k)}=p>0$. We also provide the classification of quasi-polarized indecomposable ${\rm BT}_1$-group schemes over perfect fields up to dimension $5$ via circular words. For more detailed information about the category of finite commutative group schemes over a perfect field we refer to \cite{Kraft}, \cite{Oort} and \cite{Pink}.

	\begin{subsection}{Decomposition of the Category}
	
		Let $k$ be a perfect field of characteristic $p>0$. Denote by $C=C_{k}$ the category of finite commutative group schemes over $\Spec(k)$. By definition, a finite scheme $G$ over $\Spec(k)$ is affine. Therefore category $C$ is equivalent to the category of commutative finitely generated $k$-bialgebras (which are automatically flat).

		\begin{proposition}
			The connected component $G^0$ of a zero section in a finite commutative group scheme $G$ is a closed finite commutative subgroup scheme and the quotient $G/G^0$ is \'etale.
		\end{proposition}
	
		Now we can introduce the following classes of finite commutative group schemes.

		\begin{definition}
			A finite commutative group scheme $G$ over $\Spec(k)$ is called
			\begin{itemize}
				\item{{\bf \'etale}, if the structure morphism $G \to \Spec(k)$ is \'etale,}
				\item{{\bf local}, if $G$ is connected.}
			\end{itemize}
			\end{definition}
	
		We will write $C_{loc}$ for the full subcategory of $C$ consisting of all $G \in C$ which are local and $C_{et}$ for the full subcategory of $C$ consisting of all $G \in C$ which are \'etale.
		
		An important observation is the following.
		\begin{lemma}
			Let $G \in G_{loc}$ and $H \in C_{et}$. Then
			\[
				{\rm Hom}_{C}(G,H)={\rm Hom}_{C}(H,G)=(0).
			\]
	
		\end{lemma}
	
		Each object $G \in C$ can be written in a unique way in the form
		$$
			G=G_{et}\times G_{loc},
		$$
		where $G_{et} \in C_{et}$, $G_{loc} \in C_{loc}$ and
		\[
			{\rm Hom}_{C}(G,H)={\rm Hom}_{C}(G_{et},H_{et}) \times {\rm Hom}_{C}(G_{loc},H_{loc}).
		\]
	
		The same decomposition holds for the linear dual $G^*$. Therefore category $C$ splits into four categories:
		\begin{itemize}
			\item{$C_{loc, loc},$ the category of all  $G \in C_{loc}$ with $G^D \in C_{loc}$,}
			\item{$C_{loc, et},$ the category of all  $G \in C_{loc}$ with $G^D \in C_{et}$,}
			\item{$C_{et, loc},$ the category of all  $G \in C_{et}$ with $G^D \in C_{loc}$,}
			\item{$C_{et, et},$ the category of all  $G \in C_{loc}$ with $G^D \in C_{et}$.}
		\end{itemize}

		Hence the category $C$ has the following decomposition:
		\[
			C = C_{et, et} \times C_{et, loc} \times C_{loc, et} \times C_{loc, loc}.
		\]
		Each category is Abelian and hence $C$ is Abelian itself.
	
		\begin{proposition}
			The following equivalences hold:
			\begin{description}
				\item[a]{$G \in C_{et, et}$ if and only if ${\rm Frob_G}$ and ${\rm Ver}_G$ are isomorphisms, }
				\item[b]{$G \in C_{et, loc}$ if and only if ${\rm Frob_G}$ is isomorphism and ${\rm Ver}_G$ is nilpotent,}
				\item[c]{$G \in C_{loc, et}$ if and only if ${\rm Frob_G}$ is nilpotent  and ${\rm Ver}_G$ is isomorphism,}
				\item[d]{$G \in C_{loc, loc}$ if and only if ${\rm Frob_G}$ and ${\rm Ver}_G$ are both nilpotents.}
			\end{description}
	
	\end{proposition}
	
		In this paper we focus on the group scheme ${\rm A}[p]$, where ${\rm A}$ is an Abelian variety defined over a finite field $\FF_{q}$, with $p = {\rm char}(\FF_{q})$. The group scheme ${\rm A}[p]$ belongs to category $C$. Furthermore, due to the proposition above and based on the relations
	${\rm Frob}\cdot{\rm Ver}={\rm Ver}\cdot{\rm Frob}=p$ and the fact that $p$ is nilpotent, it follows that
		$$
			A[p] \in C_{et, loc} \times C_{loc, et} \times C_{loc, loc}.
		$$
	
	\end{subsection}

	\begin{subsection}{Circular Words}
	
		We write $C(1)_{k}$ for the category of finite commutative $k-$group schemes which are killed by $p$, where $k$ is a perfect algebraically closed field of characteristic  $p>0$.  The Dieudonn\'e functor shows that the full subcategory $C(1)_{k}$ of $C$ is equivalent to the category of triples $({\rm M},\, {\mathcal F},\, {\mathcal V}),$ where
		\begin{itemize}
			\item{${\rm M}$ is a finite dimensional $k$-vector space,}
			\item{${\mathcal V}: {\rm M} \rightarrow {\rm M}$ is a ${\rm Frob}_{k}$-linear map,}
			\item{${\mathcal F}: {\rm M} \rightarrow {\rm M}$ is a ${\rm Frob}_{k}^{-1}$-linear map,}
		\end{itemize}
		such that  ${\mathcal F}{\mathcal V}={\mathcal V}{\mathcal F}=0$.
				
		\begin{definition}
			A finite locally free commutative group scheme $G$ over a scheme $S$ is called a {\bf truncated Barsotti-Tate group of level $1$} or  a {\bf ${\rm BT}_{1}$ group scheme}, if it is killed by $p$ and ${\rm Ker}({\rm Frob}_{G})={\rm Im}({\rm Ver}_{G}).$
		
		\end{definition}
		
		In terms of exact sequences, a group scheme ${\rm G} \in C(1)_{k}$ is a {\bf ${\rm BT}_{1}$} group scheme if and only if the sequence
		$$
			\xymatrix{{\rm G}   \ar[r]^{\mathcal{F}_{\rm G}} &{\rm G}^{(p)} \ar[r]^{\mathcal{V}_{\rm G}} & {\rm G} }
		$$
		is exact. For the Dieudonn\'e module this means that ${\rm Ker}(\mathcal{V}) ={\rm Im}(\mathcal{F})$ and ${\rm Ker}(\mathcal{F}) ={\rm Im}(\mathcal{V})$.
		
		In the unpublished manuscript \cite{Kraft}, Kraft  showed that there is a normal form of an object of $C(1)_{k}$. The normal form is distinguished into two types of group schemes, linear and circular types, and the group scheme ${\rm A}[p]$ corresponds only to the circular type.

		\begin{definition}
			A {\bf circular word} is a finite ordered set of symbols $\mathcal{F}$ and $\mathcal{V}$:
			$$
				w=L_1 \ldots L_{t}, \qquad L_i \in \{\mathcal{F},\, \mathcal{V}\}.
			$$

		\end{definition}
		
		We say that two words $w_1$ and $w_2$ are equivalent if there is a cyclic permutation which transforms one word into another. In other words, the class of a word is given by $[L_1 \ldots L_{t}]=[L_2 \ldots L_{t}\, L_1]=\ldots=[L_t\, L_1 \ldots L_{t-1}]$.

		A \textbf{dual word} $\bar{w}$ to the circular word $w$ is given by replacing $\mathcal{F}$ to $\mathcal{V}$ and $\mathcal{V}$ to $\mathcal{F}$ in $w$. We will study only self-dual circular words because the scheme ${\rm A}[p]$ is symmetric.

		For a given word $w$ one can construct a finite group scheme $G_{w}$ over $k$ defined by the $k-$vector space
		\[
			\mathbb{D}(G_{w})=\sum_{i=1}^{t} k\,z_{i},
		\]
		with the structure of the Diuedonn\'e module given by
		\begin{center}
			\begin{tabular}{l l}
				if $L_i=\mathcal{F}$, & then $\mathcal{F} z_{i}=z_{i+1}$ and $ \mathcal{V}\, z_{i+1}=0$,\\
				if $L_i=\mathcal{V}$, & then $\mathcal{V} z_{i+1}=z_{i}$ and $ \mathcal{F}\, z_{i}=0$.\\
			\end{tabular}
		\end{center}

		A word $w$ is called \textbf{decomposable} if there exist $t', \mu \in \ZZ_{>0}$, with $\mu \cdot t' = t$ and $L_1, ..., L_{t'} \in \{ \mathcal{F}, \mathcal{V} \}$, such that
		\begin{equation*}
		  [L_1...L_t] = [(L_1...L_{t'})^{\mu}] =
		  [(L_1...L_{h'})...(L_1...L_{h'})].
		\end{equation*}
		If such a way of writing with $\mu > 1$ is not possible, we say the word $w$ is \textbf{indecomposable}.

		Category $C(1)_{k}$ is Abelian, and all objects have a finite length. Hence  every object from  $C(1)_{k}$ is a direct sum of indecomposable objects. Up to isomorphism and permutation of the factors, this decomposition is unique. The following theorem can be found in \cite{Kraft}.

		\begin{theorem}\label{FiniteGroupSchemesAndCircularWords}
			\begin{enumerate}
				\item A circular word $w$ defines a ${\rm BT}_1$-group scheme $G_{w}$, and $w$ is indecomposible if and only if $G_{w}$ is indecomposible.
				\item Any ${\rm BT}_1$-group scheme over $k$ is a direct sum of indecomposable ${\rm BT}_1$-group schemes.
				\item For any indecomposible ${\rm BT}_1$-group scheme $G$ over $k$, there exists an indecomposible word $w$ such that $G \cong G_{w}$.
			\end{enumerate}
		
		\end{theorem}
		
		\begin{corollary}\label{IndecomposableQuasiPolarizedBT1GroupSchemes}
		  A quasi-polarized ${\rm BT}_1$-group scheme $G$ over $k$ is indecomposable (i.e. $G$ is not a direct sum of two quasi-polarized ${\rm BT}_1$-group schemes) if
		  \begin{itemize}
		    \item either $G = G_w$, where $w$ is an indecomposable word, for which $[w] = [\bar{w}]$,
		    
		    \item or $G = G_u \oplus G_v$, where $u$ and $v$ are distinct indecomposable words, such that $[\bar{u}] = [v]$.
		  \end{itemize}
		  
		\end{corollary}
	
	\end{subsection}

	\begin{subsection}{Description of  $p$-Rank and $a$-Number in Terms of Circular Words}
		
		Here we give the definitions of $p$-rank and $a$-number of an Abelian variety over ${\bar \FF}_p$ and describe them in terms of circular words.		
		
		\begin{definition}
			Let ${\rm A}$ be an Abelian variety of dimension $g$ over a perfect field $k,$ with ${\rm char}(k)=p>0.$ 
			\begin{itemize}
				\item A number $f({\rm A})$  is called {\bf $p$-rank} of Abelian variety ${\rm A}$ if 
		  	\begin{equation*}
		    	{\rm A}[p](\bar{\FF}_p) \cong
		    	(\ZZ / p \ZZ)^{f({\rm A})}.
		  	\end{equation*}
				\item A number
		  	$$
		    	a({\rm A}) =
		    	\dim_{\bar{\FF}_p} {\rm Hom}(\alpha_p, {\rm A}[p])
		  	$$
		  	is called {\bf $a$-number} of Abelian variety ${\rm A}$.
			\end{itemize}
		
		\end{definition}
		
		As the group scheme $\mu_p$ is dual to $\ZZ / p \ZZ$, and  ${\rm A}[p]$ is a self-dual group scheme, it is clear that
		\begin{equation*}
			f({\rm A}) =
			\dim_{{\bar \FF}_p} {\rm Hom}(\mu_p, {\rm A}[p]) \qquad \mbox{and} \qquad 0 \leq f({\rm A}) \leq {\rm dim}(A).
		\end{equation*}

		Since ${\rm Ker}({\rm Frob}_{k}) \cap {\rm Ker}({\rm Ver}_{k})$ is the product of a certain number of copies of $\alpha_p$, we also obtain that
		$$
			a({\rm A}) =
			\log_p {\rm ord} S =
		  \dim_{\bar{\FF}_p} \ker \left(
		    {\rm Ver}_{k} :
		    H^0({\rm A}, \Omega^1_{\rm A})
		    \rightarrow
		    H^0({\rm A}, \Omega^1_{\rm A})
		  \right),
		$$
		where $S$ is a maximal subgroup scheme in ${\rm A}[p]$ that is killed by the action of ${\rm Frob}_{k}$ and ${\rm Ver}_{k}$.

		\begin{remark}
			From the fact that the group schemes $\mu_p$ and $\alpha_p$ are simple, it follows that the invariants $p$-rank and $a$-number are additive, i.\ e.\ if $G_1$ and $G_2$ are from $C(1)_{k}$, then
		  \begin{gather*}
		    f(G_1 \oplus G_2) =
		    f(G_1) + f(G_2), \\
		    a(G_1 \oplus G_2) =
		    a(G_1) + a(G_2).
		  \end{gather*}

		\end{remark}
		
Here we provide an important result needed for computing the invariants.
		\begin{proposition}\label{aNumberCalc}
			Let $G$ be a finite group scheme over $\bar{\FF}_p$ of local-local type. Then the group scheme $\ker ({\rm Frob}_{G}) \cap \ker ({\rm Ver}_{G})$ is isomorphic to the $a(G)$ copies of $\alpha_p$, that is
			\begin{equation*}
				\alpha^{a(G)}_p \cong
				\ker ({\rm Frob}_{G}) \cap \ker ({\rm Ver}_{G}),
			\end{equation*}
			where ${\rm Frob}_{G}$ and ${\rm Ver}_{G}$ are the morphisms of Frobenius and Verschiebung acting on $G$, respectively.
	
		\end{proposition}
		
		In order to describe the $a$-number in terms of circular words, for each circular word $w$ we introduce a new auxiliary invariant, the number $c(w)$ corresponding to the word $w$.

		\begin{definition}
		  Let $w$ be a circular word, and the word $w^*$ is received from $w$ by cyclic permutation of symbols such that word $w^*$ is started from the symbol $\mathcal{F}$ and ended by the symbol $\mathcal{V}$. Then we denote by $c(w)$ a number of subwords of the form $\mathcal{FV}$ in the word $w^*$. Note that the chosen representative $w^*$ of the class $[w]$ always exists, but it may not be unique with the desired property. However, in any case the number $c(w)$ will not depend on the choice of representative $w^*$.
		  
		\end{definition}
		
Now we can describe an $a$-number in terms of $c(w)$ numbers.

		\begin{proposition}\label{ComputationOfa-number}
			Let ${\rm A}$ be an Abelian variety over $\bar{\FF}_p.$ Assume that  the group scheme  ${\rm A}[p]$ corresponds to a set of indecomposible words $\{ w_i \}$ under the equivalence of the categories. Then $a$-number of ${\rm A}$ is a sum of $c(w_i),$ i.\ e.\
		  \begin{equation*}
		    a({\rm A}) = \sum_{i = 1}^{n} c(w_i).
		  \end{equation*}
		\end{proposition}
		\begin{proof}
		  If  ${\rm A}[p]$ is a direct sum  $G_1 \oplus ... \oplus G_m$ of indecomposible group schemes, then $a(A) = \sum a(G_i)$. Hence it is sufficient to show that $a(G) = c(w)$ for an indecomposable group scheme $G = G_w$ (corresponding to an indecomposible word $w$).

			A given indecomposible word $w$ describes the operators $\mathcal{F}$ and $\mathcal{V}$ on $\DD(G) = \sum \bar{\FF}_p z_i$ under the equivalence of the categories of $C_{loc,\, loc}$ and the Diedounne modules. Therefore
		  \begin{gather*}
		    a(G) =
		    \dim_{\bar{\FF}_p}(\ker \mathcal{F} \cap \ker \mathcal{V}) = 
		    \dim_{\bar{\FF}_p} \{ x \in \DD(G) | \mathcal{F}(x) = \mathcal{V}(x) = 0 \} = \\
		  =\dim_{\bar{\FF}_p} \{ x = \sum^t_{i = 1} x_i z_i | \sum x^p_i \mathcal{F}(z_i) = \sum x^{p^{-1}} \mathcal{V}(z_i) = 0 \} =\\ 
		   = \# \{ i | \mathcal{F}(z_i) \neq 0 \text{ and } \mathcal{V}(z_i) \neq 0 \} =\\
		  =  \text{``number of subwords of the form $\mathcal{FV}$ in the word $w^*$''} = c(w).
		  \end{gather*}

		\end{proof}
		
	\end{subsection}

	\begin{subsection}{Quasi-Polarized Indecomposable ${\rm BT}_1$-Group Schemes of Order $p^{2g}$ with $g$ up to $5$}
		
		Due to the corollary \ref{IndecomposableQuasiPolarizedBT1GroupSchemes}, we can describe each isomorphism class of a quasi-polarized indecomposable ${\rm BT}_1$-group schemes over $\bar{\FF}_p$ in terms of self-dual sets of indecomposable circular words. Here we provide a complete list of isomorphism classes of quasi-polarized indecomposable ${\rm BT}_1$-group schemes over $\bar{\FF}_p$ of order $p^{2g}$ with $g$ up to $5$.
		
		\begin{center}
		\begin{tabular}{|c|c|c|c|}
		\hline
		$g$ & {\rm circular words} & {\rm group scheme} & $a$-number \\
		\hline
		
		\multirow{2}{*}{$1$} & $[ \mathcal{F} ], [ \mathcal{V} ]$ & $\mu_p \times \ZZ / p\ZZ$ & $0$ \\ \cline{2-4}
		& $[ \mathcal{FV} ]$ & ${\rm I}_{1, 1}$ & $1$ \\
		\hline
		
		$2$ & $[ \mathcal{FFVV} ]$ & ${\rm I}_{2, 1}$ & $1$ \\
		\hline
		
		\multirow{2}{*}{$3$} & $[ \mathcal{FFFVVV} ]$ & ${\rm I}_{3, 1}$ & $1$ \\ \cline{2-4}
		& $[ \mathcal{FFV} ], [ \mathcal{VVF} ]$ & ${\rm I}_{3, 2}$ & $2$ \\
		\hline
		
		\multirow{3}{*}{$4$} & $[ \mathcal{FFFFVVVV} ]$ & ${\rm I}_{4, 1}$ & $1$ \\ \cline{2-4}
		& $[ \mathcal{FFFV} ], [ \mathcal{VVVF} ]$ & ${\rm I}_{4, 2}$ & $2$ \\ \cline{2-4}
		& $[ \mathcal{FFVFVVFV} ]$ & ${\rm I}_{4, 3}$ & $3$ \\
		\hline
		
		\multirow{6}{*}{$5$} & $[ \mathcal{FFFFFVVVVV} ]$ & ${\rm I}_{5, 1}$ & $1$ \\ \cline{2-4}
		& $[ \mathcal{FFFFV} ], [ \mathcal{VVVVF} ]$ & ${\rm I}_{5, 2}$ & $2$ \\ \cline{2-4}
		& $[ \mathcal{FFFVV} ], [ \mathcal{VVVFF} ]$ & ${\rm J}_{5, 2}$ & $2$ \\ \cline{2-4}
		& $[ \mathcal{FFFVFVVVFV} ]$ & ${\rm I}_{5, 3}$ & $3$ \\ \cline{2-4}
		& $[ \mathcal{FFVFFVVFVV} ]$ & ${\rm J}_{5, 3}$ & $3$ \\ \cline{2-4}
		& $[ \mathcal{FFVFV} ], [ \mathcal{VVFVF} ]$ & ${\rm I}_{5, 4}$ & $4$ \\
		\hline
		\end{tabular}
		\end{center}
		
		In this table by the ${\rm I}_{g, a}$ we denote a quasi-polarized indecomposable ${\rm BT}_1$-group scheme of order $p^{2g}$ with $a$-number $a$. Moreover, for $g=5$ there are two pairs of different isomorphism classes of the ${\rm BT}_1$-group scheme of the same order and $a$-number.

		Let us consider the case with $g = 2$ as an example. The corollary \ref{IndecomposableQuasiPolarizedBT1GroupSchemes} shows that the indecomposable ${\rm BT}_1$-group scheme $G$ is either $G_w$ for the indecomposable word $w$ of length $4$ with $w = \bar{w}$, or $G = G_u \oplus G_v$, where $u$ and $v$ are distinct indecomposable words of length $2$ such that $\bar{u} = v$. However, there are only  two indecomposable words of length $2$, up to the equivalence, namely the words $\mathcal{FV}$ and $\mathcal{VF}.$ Since they are equivalent to each other, it remains only to iterate over all indecomposable words of length $4$ for which $[w] = [\bar{w}]$.

		The set of non-equivalent words of length $4$ is $w_1 = [ \mathcal{FFFF} ]$, $w_2 = [ \mathcal{FFFV} ]$, $w_3 = [ \mathcal{FFVV} ]$, $w_4 = [ \mathcal{FVFV} ]$, $w_5 = [ \mathcal{FVVV} ]$ and $w_6 = [ \mathcal{VVVV} ]$, and only the words $w_2$, $w_3$ and $w_5$ are indecomposable. Moreover, there is unique the word $w_3$  which is self-dual and therefore only this word corresponds to the quasi-polarized indecomposable ${\rm BT}_1$-group scheme, which we denote by ${\rm I}_{2, 1}$.
		
		For all rows of the table, the $a$-number is computed be use of the proposition \ref{ComputationOfa-number}, and one can observe that if $g = 5$, the situation arises repeatedly in which for one pair $(g, a)$ there are two non-isomorphic group schemes that correspond to the given parameters. It leads us to introducing additional notations to distinguish these group schemes.
		
	\end{subsection}

\end{section}

\begin{section}{Algorithm Foundation}\label{Algorithm_Foundation}

For the sake of completeness, let us recall the result from \cite{Zaytsev}~(section~4) and adjust it for our use.
		
	In this section, we recall the results of \cite{Zaytsev} demonstrating that the choice of a decomposition group and CM type derives a decomposition of the ${\rm BT}_1$-group scheme of a simple Abelian variety into irreducible ${\rm BT}_1$-group schemes. The scheme ${\rm A}[p]$ arises after the reduction of a CM Abelian variety $\mathcal{A}$ at a place of good reduction. Here we will develop an explicit representation theory on the basis of the approach of \cite{Ekedahl}~(section~2).
	
	Let $\mathcal{A}$ be an Abelian  scheme of relative dimension $g$ over ${\rm Spec}(\cO_{L})$, where $\cO_{L}$ is the full ring of integers of a number field $L$. Assume that $\mathcal{A}$ has a complex multiplication by the full ring of integers of a CM field $K$ and  $K \subset L$. Suppose that $\mathcal{A}$ has a good reduction at a prime ideal $\mathcal{B} \subset \cO_{L}$ and a number prime $p \in \mathcal{B} \cap \ZZ$ is unramified in $K$.

	Let $p\cO_{K}=P_{1} \ldots P_{m}$ be a decomposition into distinct prime ideals. Let $\tilde{K}$ be the Galois closure of $K$ over $\QQ$  and $p\cO_{K}= \tilde{P}_{1} \ldots \tilde{P}_{l}$ be the decomposition into  prime ideals in $\tilde{K}$ (note that $p$ is also unramified in the Galois closure $\tilde{K}$, since it is a composite of all embeddings of $K$ into fixed algebraic closure of $\QQ$).

	The ring $\cO_{K}$ is a free $\ZZ-$module and $\cO_{K} \otimes_\ZZ \QQ \cong K.$ The semi-simple $\bar{\QQ}$-algebra $\cO_{K} \otimes_\ZZ \bar{\QQ}$ can be decomposed into irreducible components
	\begin{displaymath}
		\cO_{K} \otimes_\ZZ \bar{\QQ}\cong\prod_{\alpha \in {\rm Hom}(K, \bar{\QQ})} \bar{\QQ},
	\end{displaymath}
	and hence it induces a decomposition into simple $\cO_{K} \otimes_\ZZ \bar{\FF}_p$-modules
	\begin{displaymath}
		\cO_{K} \otimes_\ZZ \bar{\FF}_p\cong\prod_{\alpha \in {\rm Hom}(K, \bar{\QQ})} \bar{\FF}_p,
	\end{displaymath}
	since $p$ is unramified in $K$. So the representation on $\bar{\mathbb{F}}_q$-vector space ${\mathbb D}(A[p])$ of the semi-simple algebra
	$\cO_{K}\otimes_\ZZ \bar{\FF}_p$   is a direct sum of irreducible representations
	\begin{displaymath}
		{\mathbb D}(A[p]) \cong \bigoplus_{\alpha \in {\rm Hom}(K, \CC)} V_{\alpha},
	\end{displaymath}
	where $V_{\alpha}$ is an irreducible $\cO_{K} \otimes_\ZZ \bar{\FF}_p$-module.

	In order to obtain an explicit description of the circular words, we reformulate this decomposition in terms of Galois actions. Let $G$ be a Galois group ${\rm Gal}(\tilde{K}/ \QQ)$, $\Delta={\rm Gal}(\tilde{K}/K)$, $\tilde{P}=\mathcal{B} \cap \cO_{\tilde{K}}$  and $\sigma$ be a generator of the decomposition group of a prime ideal $\tilde{P}$ in $\tilde{K}$. Then the decomposition of $\cO_{K} \otimes_\ZZ \bar{\FF}_p$ can be written as follows,
	\begin{gather*}
		\cO_{K}/p\cO_{K} \otimes_\ZZ \bar{\FF}_p\cong
		\bigoplus_{i=1}^{m} F_{P_i} \otimes_{\FF_p} \bar{\FF}_p \cong \\
		\bigoplus_{i=1}^{m}\left( \bigoplus_{\alpha \in {\rm Hom}(F_{P_i}, \bar{\FF}_p)}  \bar{\FF}_p\right)
	 	\cong\prod_{\alpha \in {\rm Hom}(K, \bar{\QQ})} \bar{\FF}_p,
	\end{gather*}
	where $F_{P_i}$ is the residue field of a prime ideal $P_i$. The last isomorphism comes from the fact that the embeddings $F_{P_i} \to \bar{\FF}_p$ are in one-to-one correspondence to embeddings $K \to \bar{\QQ}$ (since $p$ is unramified).
		
	Let us fix a prime ideal  $\tilde{P}=\tilde{P}_i$  and an isomorphism $\cO_{\tilde{K}}/ \tilde{P} \cong \FF_{q} \subset \bar{\FF}_{p}$. Then each $\alpha \in G$ induces  an embedding $\FF_{q}$ into  $\bar{\FF}_{p}$ by sending $(a \,{\rm mod } \, \tilde{P}) \mapsto (\alpha(a) \,{\rm mod } \, \tilde{P})$. Then we have the following decomposition of ${\mathbb D}(A[p])$ into $2g$ one-dimensional eigenspaces,
	$$
		{\mathbb D}(A[p])=\bigoplus_{\alpha \in G \setminus \Delta} {\rm V}_{\alpha},
	$$
	where $\alpha$ runs through all conjugate classes of $G$ by action of $\Delta$ on the right and
	$$
		{\rm V}_{\alpha}=\{v \in {\mathbb D}(A[p]) \, | \, a(v)= (\alpha(a) \,{\rm mod } \, \tilde{P}) v\qquad \mbox{for any}\, a \in \cO_{K}    \}.
	$$
	Since all $V_{\alpha}$ are isomorphic to each other, it follows that ${\rm dim}_{\bar{\FF}_p}V_{\alpha}=1$ for each $\alpha \in G \setminus \Delta$.

	The action of $\cO_K$ on ${\mathbb D}(A[p])$ is imposed by a fixed isomorphism $\cO_{K} \cong {\rm End}(A)$. The Frobenius ${\rm Fr}$ on ${\mathbb D}(A[p])$ is $p$-linear and commutes with the  $\cO_K$-action. Hence for each irreducible component $V_{\alpha}$ of ${\mathbb D}(A[p])$ there exists an irreducible component $V_{\beta}$ such that ${\rm Fr}(V_{\alpha}) \subset V_{\beta}$. Moreover, $\beta$ corresponds to the class of $\sigma \alpha$ in $G \setminus \Delta$ (since $\sigma(a) \equiv a^p\, \mbox{mod}\, \tilde{P}$). In other words,
	$$
		{\rm Fr}:{\rm V}_{\alpha} \rightarrow {\rm V}_{\sigma \alpha}.
	$$
		
	Let us denote the fibre product $(\mathcal{A} \, {\rm mod}\, \mathcal{B}) \times \bar{\FF}_{p}$ by ${\rm A}$. The set $S$ of  the isomorphism classes of irreducible factors of ${\mathbb D}(A[p])$ as an $\cO_K\otimes_{\ZZ} \bar{\FF}_p$-module can be identified with the set $\Hom(K, \bar{\QQ})$ via identification of the isomorphism classes of  irreducible representations of $\cO_K\otimes_{\ZZ} \bar{\FF}_p$ in $\bar{\FF}_p$ and the set $S$. There is an exact sequence of group schemes
	$$
		0 \rightarrow
		A[{\rm Ver}]
		\rightarrow
		A[p]
		\xrightarrow{\rm Ver}
		A[{\rm Fr}]
		\rightarrow 0,
	$$
	which yields an exact sequence of the Dieudonn\'e modules
	$$
		0 \rightarrow
		{\mathbb D}(A[{\rm Ver}])
		\rightarrow
		{\mathbb D}(A[p])
		\xrightarrow{\mathcal F}
		{\mathbb D}(A[{\rm Fr}])
		\rightarrow 0.
	$$

	Due to this exact sequence, the set $S$ is a disjoint union of $S^0$ and $S^1$, where $S^0$ and $S^1$ are the classes of irreducible representations occurring in ${\mathbb D}(A[{\rm Ver}])$ and ${\mathbb D}(A[{\rm Fr}])$, respectively.
	
	Let $S^{1}$ be a CM type of $\mathcal{A}$ and $S^{0}$ be the conjugation of the CM type. On the basis of these data we can draw a graph of a circular word for ${\rm BT}_1$ of an Abelian variety ${\rm A}$ over $\bar{\FF}_{p}.$ Vertices of $\Gamma$ are classes $G / \Delta$, and there is an arrow between a class $[\alpha] \in G / \Delta$ and $[\sigma \alpha]$, where $\sigma$ is a generator of the decomposition group. The arrow  $[\alpha]\to[\sigma \alpha]$ is labeled by $\mathcal{V}$  if $[\alpha] \in S^1$ and  $[\sigma \alpha]\to[\alpha]$ is labeled by $\mathcal{F}$ otherwise.
	
\end{section}

\begin{section}{Algorithm Description}\label{Algorithm_Description}

	\begin{subsection}{The Algorithm}
		
		Let all the notations be defined as in the previous section.
		
		For convenience, we introduce the notion characterizing the decomposition of a prime number $p$ in a CM field $K$. Let $K_0$ be a totally real subfield of $K$. By the \textbf{decomposition type} of $p$ in $K$ we mean a pair of non-negative integers $(\alpha, \beta)$ such that $p\cO_K = \mathcal{P}_1 ... \mathcal{P}_{\alpha}$ and $p\cO_{K_0} = \mathcal{Q}_1 ... \mathcal{Q}_{\beta}$, where $\mathcal{P}_i$ and $\mathcal{Q}_j$ are prime ideals (not necessarily distinct) that lie in $\cO_K$ and $\cO_{K_0}$ respectively.
	
		
		Now we describe an algorithm that searches for all possible non-isomorphic decompositions of the group scheme ${\rm A}[p]$ with the given Galois group $G = {\rm Gal}(\tilde{K} / \QQ)$ and the decomposition type $(\alpha, \beta)$ of prime $p$. According to the previous section, the algorithm describes actions of the $\mathcal{F}$ and $\mathcal{V}$ on the proper subspaces of the module $\mathbb D (A[p])$. The algorithm realization in the GAP computer algebra system can be found in \cite{Program}.
		
		
		Using pseudocode, the algorithm can be written as follows.
		
		\begin{algorithm}
		\caption{}\label{alg:Algorithm1}
		\begin{algorithmic}[1]
		
		\State $G = {\rm Gal}(\tilde{K} / \QQ)$
		\State $\iota List \gets {\rm GetInvolutions}(G)$
		\ForAll{$\iota \in \iota List$}
			\State $DeltaList \gets {\rm GetDeltaSubgroups}(G, \iota)$
			\ForAll{$\Delta \in DeltaList$}
				\State $H_0 \gets {\rm GetH_0Subgroup}(G, \iota, \Delta)$
				\State $CMTypesList \gets {\rm GetCMTypes}(G, \iota, \Delta)$
				\ForAll{$S^1 \in CMTypesList$}
					\ForAll{$\sigma \in G$}
						\State $\alpha \gets | \Delta \setminus G / \langle \sigma \rangle |$
						\State $\beta \gets | H_0 \setminus G / \langle \sigma \rangle |$
						\State $Words \gets GetWords(G, \iota, \Delta, S^1, \sigma)$
						\State \textbf{Print: } $(\alpha, \beta) \longrightarrow Words$
					\EndFor
				\EndFor
			\EndFor
		\EndFor
		
		\end{algorithmic}
		\end{algorithm}
		
		
		The algorithm begins with the fixation of an involution $\iota$ that belongs to the center of $G$. The element $\iota$ induces a complex conjugation of the CM field $K$. In this case, according to the fundamental theorem of Galois theory, the field $K$ corresponds to a subgroup $\Delta \subset G$ of order $|G|/(2g)$ that does not contain $\iota$. Besides, $\Delta$ could be a normal subgroup in $G$ if and only if $\Delta = \langle 1 \rangle$, i.e. when $K / \QQ$ is Galois.
	
		The subgroup $H_0$ corresponds to the subfield $K_0 \subset K$. Hence it could be constructed as the smallest subgroup that contains $\iota$ and $\Delta$.

		We select a primitive CM-type $S^1$ from all the CM-types that arise with given 
$G$, $\iota$ and $\Delta$ so the pair $(K, S^1)$ uniquely corresponds to a simple Abelian variety ${\rm A}$ over $K$.

		Finally we should choose a prime ideal $\tilde{\mathcal{P}} \subset \cO_{\tilde{K}}$ for which we want to establish correspondences between the decomposition types of $p = \tilde{\mathcal{P}} \cap \ZZ$ in $K$ and the decompositions of the group scheme ${\rm A}[p]$. For this, it suffices to fix the decomposition group $\mathcal{D}$ of the ideal $\tilde{\mathcal{P}},$ which is  a cyclic subgroup of $G$. Note that it could lead us to extra connections since the algorithm runs all possible $\sigma \in G$ which generate cyclic subgroups of $G$ and it doesn't exclude the situation in which the subgroup $\langle \sigma \rangle \subset G$ doesn't arise as a decomposition group of some $\tilde{\mathcal{P}}$. Therefore, it is important to remember that the algorithm only eliminates the impossible variants of the decomposition of the scheme ${\rm A}[p]$ for the given decomposition type of $p$ and it doesn't guarantee that all remaining decompositions are realisable for some Abelian variety $\mathcal{A}$ and prime $p$.
			
	\end{subsection}

	\begin{subsection}{The Main Result}
		
		
		We have applied the algorithm in order to explicitly obtain a full table of all the possible decompositions of the scheme ${\rm A}[p]$ with the given decomposition type $(\alpha, \beta)$ of prime $p$ in CM field $K$ for the dimensions $g = 1, 2, 3, 4, 5$.
		
		It is known that for any $g > 0$ there is only a finite number of possible Galois groups $G = {\rm Gal}(\tilde{K} / \QQ)$ such that $\tilde{K}$ is a Galois closure of an CM field $K$ of dimension $2g$ over $\QQ$. The paper \cite{Dodson} contains the complete lists of such groups for $g \leq 7$. Therefore, we can use the algorithm for each group from the list to obtain a set of all possible decompositions of the scheme ${\rm A}[p]$ that correspond to decompositions of prime $p$ that can arise for a simple Abelian variety of the given dimension. Recall that the algorithm will return decompositions of the ${\rm A}[p]$ as sets of circular words, so to obtain the final result we should use the table at the end of section \ref{Preliminaries_and_Classification}.

		The list of all the arising Galois groups for the $\tilde{K}$ with $K$ of dimension $g$ is presented in the following table. For brevity's sake, we denote the groups as in the GAP computer algebra system. So each group of order $n$ could be written as $G_{n, m}$, where $m$ is the second index of the group in the GAP Small Groups Library.
		
		\begin{center}
		\begin{longtable}{|c|c|c|}
		\hline
		
		{$\dim \mathcal{A}$} & {$|G|$} & {\sf groups list} \\
		\hline
		
		$2$ & $2$ & $G_{2, 1}$ \\
		\hline
		
		\multirow{2}{*}{$4$} & $4$ & $G_{4, 1}$ \\ \cline{2-3}
		& $8$ & $G_{8, 3}$ \\
		\hline
		
		\multirow{4}{*}{$6$} & $6$ & $G_{6, 2}$ \\ \cline{2-3}
		& $12$ & $G_{12, 4}$ \\ \cline{2-3}
		& $24$ & $G_{24, 13}$ \\ \cline{2-3}
		& $48$ & $G_{48, 48}$ \\
		\hline
		
		\multirow{10}{*}{$8$} & $8$ & $G_{8, 1}$; $G_{8, 2}$; $G_{8, 3}$; $G_{8, 4}$; $G_{8, 5}$ \\ \cline{2-3}
		& $16$ & $G_{16, 3}$; $G_{16, 6}$; $G_{16, 7}$; $G_{16, 8}$; $G_{16, 11}$; $G_{16, 13}$ \\ \cline{2-3}
		& $24$ & $G_{24, 3}$; $G_{24, 13}$ \\ \cline{2-3}
		& $32$ & $G_{32, 6}$; $G_{32, 7}$; $G_{32, 11}$; $G_{32, 27}$; $G_{32, 43}$; $G_{32, 49}$ \\ \cline{2-3}
		& $48$ & $G_{48, 29}$; $G_{48, 48}$ \\ \cline{2-3}
		& $64$ & $G_{64, 32}$; $G_{64, 34}$; $G_{64, 134}$; $G_{64, 138}$ \\ \cline{2-3}
		& $96$ & $G_{96, 204}$ \\ \cline{2-3}
		& $128$ & $G_{128, 928}$ \\ \cline{2-3}
		& $192$ & $G_{192, 201}$; $G_{192, 1493}$ \\ \cline{2-3}
		& $384$ & $G_{384, 5602}$ \\
		\hline
		
		\multirow{10}{*}{$10$} & $10$ & $G_{10, 2}$ \\ \cline{2-3}
		& $20$ & $G_{20, 4}$ \\ \cline{2-3}
		& $40$ & $G_{40, 12}$ \\ \cline{2-3}
		& $120$ & $G_{120, 35}$ \\ \cline{2-3}
		& $160$ & $G_{160, 235}$ \\ \cline{2-3}
		& $240$ & $G_{240, 189}$ \\ \cline{2-3}
		& $320$ & $G_{320, 1636}$ \\ \cline{2-3}
		& $640$ & $G_{640, 21536}$ \\ \cline{2-3}
		& $1920$ & $G_{1920, 240997}$ \\ \cline{2-3}
		& $3840$ & $G_{2, 1} \times G_{1920, 240996}$ \\
		\hline
		\end{longtable}
		\end{center}
		
		
		Let us look at the algorithm in more detail via the following example. Let $g = 5$ and $G = {\rm Gal}(\tilde{K} / \QQ) = G_{40, 12}$. With accuracy up to an isomorphism, we can assume that $G$ is generated by permutations $(2, 7)(3, 4, 8, 9)$ and $(1, 4, 3, 8)$ as a subgroup of order $40$ of the symmetric group $S_{10}$.
		
		The only non-trivial involution lying in the center of $G$ is the permutation ${\iota = (2, 7)}$.
		
		There are only $10$ subgroups of $G$ suitable to serve as $\Delta$. But, up to an automorphism of $G$ that preserves $\iota$, there is only one such subgroup. Therefore we would consider subgroup $\langle \delta \rangle$ as $\Delta$ where ${\delta = (3, 4, 8, 9)}$ has order $4$ in $G$. So the $H_0$ subgroup would be generated by the elements $\iota$ and $\delta$.
		
		The quotient $\Delta \setminus G$ consists of
		\begin{gather*}
			\{
				\Delta,
				\Delta (2, 7),
				\Delta (1, 3)(4, 8),
				\Delta (1, 3)(2, 7)(4, 8),
				\Delta (1, 4, 9, 3),
				\Delta (1, 4, 9, 3)(2, 7), \\
				\Delta (1, 8, 9, 4, 3),
				\Delta (1, 8, 9, 4, 3)(2, 7),
				\Delta (1, 9, 8, 3),
				\Delta (1, 9, 8, 3)(2, 7)
			\}
		\end{gather*}
		There are $32$ CM-types that can be constructed from the elements of this set. But only $16$ of them would not be pairwise conjugated. Besides, there is one primitive CM-type in this set and we don't need to examine it. Furthermore, if we consider only the CM-types which are not translated to each other by the automorphisms of $G$ then there are only $5$ different CM-types that are worth considering, namely:

		\begin{enumerate}
			\item[\textbf{(A)}] $S^1 = \{ \Delta, \Delta (1, 3)(4, 8), \Delta (1, 4, 9, 3), \Delta (1, 8, 9, 4, 3), \Delta (1, 9, 8, 3)(2, 7) \}$;
			\item[\textbf{(B)}] $S^1 = \{ \Delta, \Delta (1, 3)(4, 8), \Delta (1, 4, 9, 3), \Delta (1, 8, 9, 4, 3)(2, 7), \Delta (1, 9, 8, 3)(2, 7) \}$;
			\item[\textbf{(C)}] $S^1 = \{ \Delta, \Delta (1, 3)(4, 8), \Delta (1, 4, 9, 3)(2, 7), \Delta (1, 8, 9, 4, 3), \Delta (1, 9, 8, 3)(2, 7) \}$;
			\item[\textbf{(D)}] $S^1 = \{ \Delta, \Delta (1, 3)(4, 8), \Delta (1, 4, 9, 3)(2, 7), \Delta (1, 8, 9, 4, 3)(2, 7), \Delta (1, 9, 8, 3)(2, 7) \}$;
			\item[\textbf{(E)}] $S^1 = \{ \Delta, \Delta (1, 3)(2, 7)(4, 8), \Delta (1, 4, 9, 3)(2, 7), \Delta (1, 8, 9, 4, 3)(2, 7), \Delta (1, 9, 8, 3)(2, 7) \}$.
		\end{enumerate}
		
		Finally, with $G$, $\Delta \setminus G$ and $S^1$ we can construct the correspondence $${(\alpha, \beta)} \rightarrow \text{*\textit{set of circular words}*}$$ for each element ${\sigma \in G}$ using the explicit formulas from the algorithm and the end of section \ref{Algorithm_Foundation}. It leads us to the following table:
		
		\begin{center}
		\begin{longtable}{| c | c |}
		\hline
		{\sf ideal decomposition}
		& {\sf circular words} \\
		\hline
		
		\multirow{3}{*}{$\mathcal{P}$}
		& $[ \mathcal{FFFFFVVVVV} ]$ \\ \cline{2-2}
		& $[ \mathcal{FFVVFVVFFV} ]$ \\ \cline{2-2}
		& $[ \mathcal{FFFVFVVVFV} ]$ \\
		\hline
		
		\multirow{3}{*}{$\mathcal{P}\mathcal{P}^{c}$}
		& $[ \mathcal{FFFFV} ], [ \mathcal{VVVVF} ]$ \\ \cline{2-2}
		& $[ \mathcal{FFFVV} ], [ \mathcal{VVVFF} ]$ \\ \cline{2-2}
		& $[ \mathcal{FFVFV} ], [ \mathcal{VVFVF} ]$ \\
		\hline
		
		\multirow{4}{*}{$\mathcal{P}_{1}\mathcal{P}_{1}^{c}\mathcal{P}_{2}$}
		& $[ \mathcal{FV} ], [ \mathcal{F} ], [ \mathcal{V} ], [ \mathcal{F} ], [ \mathcal{V} ], [ \mathcal{F} ], [ \mathcal{V} ], [ \mathcal{F} ], [ \mathcal{V} ]$ \\ \cline{2-2}
		& $[ \mathcal{FFFV} ], [ \mathcal{VVVF} ], [ \mathcal{FV} ]$ \\ \cline{2-2}
		& $[ \mathcal{FFVV} ], [ \mathcal{FFVV} ], [ \mathcal{FV} ]$ \\ \cline{2-2}
		& $[ \mathcal{FV} ], [ \mathcal{FV} ], [ \mathcal{FV} ], [ \mathcal{FV} ], [ \mathcal{FV} ]$ \\
		\hline
		
		\multirow{4}{*}{$\mathcal{P}_{1}\mathcal{P}_{1}^{c}\mathcal{P}_{2}\mathcal{P}_{2}^{c}$}
		& $[ \mathcal{F} ], [ \mathcal{V} ], [ \mathcal{F} ], [ \mathcal{V} ], [ \mathcal{F} ], [ \mathcal{V} ], [ \mathcal{F} ], [ \mathcal{V} ], [ \mathcal{F} ], [ \mathcal{V} ]$ \\ \cline{2-2}
		& $[ \mathcal{FFFV} ], [ \mathcal{VVVF} ], [ \mathcal{F} ], [ \mathcal{V} ]$ \\ \cline{2-2}
		& $[ \mathcal{FFVV} ], [ \mathcal{FFVV} ], [ \mathcal{F} ], [ \mathcal{V} ]$ \\ \cline{2-2}
		& $[ \mathcal{FV} ], [ \mathcal{FV} ], [ \mathcal{FV} ], [ \mathcal{FV} ], [ \mathcal{F} ], [ \mathcal{V} ]$ \\
		\hline
		
		\multirow{3}{*}{$\mathcal{P}_{1}\mathcal{P}_{1}^{c}\mathcal{P}_{2}\mathcal{P}_{2}^{c}\mathcal{P}_{3}$}
		& $[ \mathcal{FV} ], [ \mathcal{F} ], [ \mathcal{V} ], [ \mathcal{F} ], [ \mathcal{V} ], [ \mathcal{F} ], [ \mathcal{V} ], [ \mathcal{F} ], [ \mathcal{V} ]$ \\ \cline{2-2}
		& $[ \mathcal{FV} ], [ \mathcal{FV} ], [ \mathcal{FV} ], [ \mathcal{F} ], [ \mathcal{V} ], [ \mathcal{F} ], [ \mathcal{V} ]$ \\ \cline{2-2}
		& $[ \mathcal{FV} ], [ \mathcal{FV} ], [ \mathcal{FV} ], [ \mathcal{FV} ], [ \mathcal{FV} ]$ \\
		\hline
		
		$\mathcal{P}_{1}\mathcal{P}_{2}\mathcal{P}_{3}\mathcal{P}_{4}\mathcal{P}_{5}$
		& $[ \mathcal{FV} ], [ \mathcal{FV} ], [ \mathcal{FV} ], [ \mathcal{FV} ], [ \mathcal{FV} ]$ \\
		\hline
		
		\multirow{3}{*}{$\mathcal{P}_{1}\mathcal{P}_{1}^{c}\mathcal{P}_{2}\mathcal{P}_{2}^{c}\mathcal{P}_{3}\mathcal{P}_{3}^{c}$}
		& $[ \mathcal{F} ], [ \mathcal{V} ], [ \mathcal{F} ], [ \mathcal{V} ], [ \mathcal{F} ], [ \mathcal{V} ], [ \mathcal{F} ], [ \mathcal{V} ], [ \mathcal{F} ], [ \mathcal{V} ]$ \\ \cline{2-2}
		& $[ \mathcal{FV} ], [ \mathcal{FV} ], [ \mathcal{F} ], [ \mathcal{V} ], [ \mathcal{F} ], [ \mathcal{V} ], [ \mathcal{F} ], [ \mathcal{V} ]$ \\ \cline{2-2}
		& $[ \mathcal{FV} ], [ \mathcal{FV} ], [ \mathcal{FV} ], [ \mathcal{FV} ], [ \mathcal{F} ], [ \mathcal{V} ]$ \\
		\hline
		
		$\mathcal{P}_{1}\mathcal{P}_{1}^{c}\mathcal{P}_{2}\mathcal{P}_{2}^{c}\mathcal{P}_{3}\mathcal{P}_{3}^{c}\mathcal{P}_{4}\mathcal{P}_{4}^{c}\mathcal{P}_{5}\mathcal{P}_{5}^{c}$
		& $[ \mathcal{F} ], [ \mathcal{V} ], [ \mathcal{F} ], [ \mathcal{V} ], [ \mathcal{F} ], [ \mathcal{V} ], [ \mathcal{F} ], [ \mathcal{V} ], [ \mathcal{F} ], [ \mathcal{V} ]$ \\
		\hline
		\end{longtable}
		\end{center}
		
		
		Thus by applying the algorithm to all Galois groups from the list above and by replacing the sets of circular words by the quasi-polarized ${\rm BT}_1$-group schemes, one obtains the following correspondences between the decomposition types of the prime $p$ and the decompositions of the scheme ${\rm A}[p]$.
		
		\begin{subsubsection}{Reproduction the result from an article \cite{Zaytsev} for $g = 1, 2, 3$}
		
		\begin{center}
		\begin{longtable}{| c | c | c | c |}
		\hline
		{\sf ideal decomposition}
		& {\sf group scheme decomposition}
		& {\sf $p$-rank}
		& {\sf $a$-number} \\
		\hline
		
		
		\multicolumn{4}{|c|}{\bf  Dimension 1}\\
		\hline
		
		$\mathcal{P}$
		& ${\rm I}_{1, 1}$
		& $0$
		& $1$ \\
		\hline
		
		$\mathcal{P}\mathcal{P}^{c}$
		& $\mu_p \times \ZZ / p\ZZ$
		& $1$
		& $0$ \\
		\hline
		
		
		\multicolumn{4}{|c|}{\bf Dimension 2}\\
		\hline
		
		$\mathcal{P}$
		& ${\rm I}_{2, 1}$
		& $0$
		& $1$ \\
		\hline
		
		\multirow{2}{*}{$\mathcal{P}\mathcal{P}^{c}$}
		& $(\mu_p \times \ZZ / p\ZZ)^2$
		& $2$
		& $0$ \\ \cline{2-4}
		& ${\rm I}_{1, 1}^2$
		& $0$
		& $2$ \\
		\hline
		
		$\mathcal{P}_{1}\mathcal{P}_{2}$
		& ${\rm I}_{1, 1}^2$
		& $0$
		& $2$ \\
		\hline
		
		$\mathcal{P}_{1}\mathcal{P}_{1}^{c}\mathcal{P}_2$
		& $(\mu_p \times \ZZ / p\ZZ) \times {\rm I}_{1, 1}$
		& $1$
		& $1$ \\
		\hline
		
		$\mathcal{P}_{1}\mathcal{P}_{1}^{c}\mathcal{P}_{2}\mathcal{P}_{2}^{c}$
		& $(\mu_p \times \ZZ / p\ZZ)^2$
		& $2$
		& $0$ \\
		\hline
		
		
		\multicolumn{4}{|c|}{\bf Dimension 3}\\
		\hline
		
		\multirow{2}{*}{$\mathcal{P}$}
		& ${\rm I}_{3, 1}$
		& $0$
		& $1$ \\ \cline{2-4}
		& ${\rm I}_{1, 1}^3$
		& $0$
		& $3$ \\
		\hline
		
		\multirow{2}{*}{$\mathcal{P}\mathcal{P}^{c}$}
		& $(\mu_p \times \ZZ / p\ZZ)^3$
		& $3$
		& $0$ \\ \cline{2-4}
		& ${\rm I}_{3, 2}$
		& $0$
		& $2$ \\
		\hline
		
		$\mathcal{P}_{1}\mathcal{P}_{2}$ &
		${\rm I}_{1, 1} \times {\rm I}_{2, 1}$
		& $0$
		& $2$ \\
		\hline
		
		\multirow{3}{*}{$\mathcal{P}_{1}\mathcal{P}_{1}^{c}\mathcal{P}_{2}$}
		& $(\mu_p \times \ZZ / p\ZZ)^2 \times {\rm I}_{1, 1}$
		& $2$
		& $1$ \\ \cline{2-4}
		& $(\mu_p \times \ZZ / p\ZZ) \times {\rm I}_{2, 1}$
		& $1$
		& $1$ \\ \cline{2-4}
		& ${\rm I}_{1, 1}^3$
		& $0$
		& $3$ \\
		\hline
		
		$\mathcal{P}_{1}\mathcal{P}_{2}\mathcal{P}_{3}$
		& ${\rm I}_{1, 1}^3$
		& $0$
		& $3$ \\
		\hline
		
		\multirow{2}{*}{$\mathcal{P}_{1}\mathcal{P}_{1}^{c}\mathcal{P}_{2}\mathcal{P}_{2}^{c}$}
		& $(\mu_p \times \ZZ / p\ZZ)^3$
		& $3$
		& $0$ \\ \cline{2-4}
		& $(\mu_p \times \ZZ / p\ZZ) \times {\rm I}_{1, 1}^2$
		& $1$
		& $2$ \\
		\hline
		
		$\mathcal{P}_{1}\mathcal{P}_{1}^{c}\mathcal{P}_{2}\mathcal{P}_{3}$
		& $(\mu_p \times \ZZ / p\ZZ) \times {\rm I}_{1, 1}^2$
		& $1$
		& $2$ \\
		\hline
		
		$\mathcal{P}_{1}\mathcal{P}_{1}^{c}\mathcal{P}_{2}\mathcal{P}_{2}^{c}\mathcal{P}_{3}$
		& $(\mu_p \times \ZZ / p\ZZ)^2 \times {\rm I}_{1, 1}$
		& $2$
		& $1$ \\
		\hline
		
		$\mathcal{P}_{1}\mathcal{P}_{1}^{c}\mathcal{P}_{2}\mathcal{P}_{2}^{c}\mathcal{P}_{3}\mathcal{P}_{3}^{c}$
		& $(\mu_p \times \ZZ / p\ZZ)^3$
		& $3$
		& $0$ \\
		\hline
		
		\end{longtable}
		\end{center}
		
		\end{subsubsection}
		
		\begin{subsubsection}{{\normalsize The new result for $g = 4$}}
		
		\normalsize
		\begin{center}
		\begin{longtable}{| c | c | c | c |}
		\hline
		{\sf ideal decomposition}
		& {\sf group scheme decomposition}
		& {\sf $p$-rank}
		& {\sf $a$-number} \\
		\hline
		
		\multirow{2}{*}{$\mathcal{P}$}
		& ${\rm I}_{4, 1}$
		& $0$
		& $1$ \\ \cline{2-4}
		& ${\rm I}_{4, 3}$
		& $0$
		& $3$ \\
		\hline
		
		\multirow{3}{*}{$\mathcal{P}\mathcal{P}^{c}$}
		& $(\mu_p \times \ZZ / p\ZZ)^4$
		& $4$
		& $0$ \\ \cline{2-4}
		& ${\rm I}_{4, 2}$,
		  ${\rm I}_{2, 1}^2$
		& $0$
		& $2$ \\ \cline{2-4}
		& ${\rm I}_{1, 1}^4$
		& $0$
		& $4$ \\
		\hline
		
		\multirow{2}{*}{$\mathcal{P}_{1}\mathcal{P}_{2}$}
		& ${\rm I}_{1, 1} \times {\rm I}_{3, 1}$,
		  ${\rm I}_{2, 1}^2$
		& $0$
		& $2$ \\ \cline{2-4}
		& ${\rm I}_{1, 1}^4$
		& $0$
		& $4$ \\
		\hline
		
		\multirow{5}{*}{$\mathcal{P}_{1}\mathcal{P}_{1}^{c}\mathcal{P}_{2}$}
		& $(\mu_p \times \ZZ / p\ZZ)^3 \times {\rm I}_{1, 1}$
		& $3$
		& $1$ \\ \cline{2-4}
		& $(\mu_p \times \ZZ / p\ZZ)^2 \times {\rm I}_{2, 1}$
		& $2$
		& $1$ \\ \cline{2-4}
		& $(\mu_p \times \ZZ / p\ZZ) \times {\rm I}_{3, 1}$
		& $1$
		& $1$ \\ \cline{2-4}
		& $(\mu_p \times \ZZ / p\ZZ) \times {\rm I}_{1, 1}^3$
		& $1$
		& $3$ \\ \cline{2-4}
		& ${\rm I}_{1, 1} \times {\rm I}_{3, 2}$,
		  ${\rm I}_{1, 1}^2 \times {\rm I}_{2, 1}$
		& $0$
		& $3$ \\
		\hline
		
		$\mathcal{P}_{1}\mathcal{P}_{2}\mathcal{P}_{3}$
		& ${\rm I}_{1, 1}^2 \times {\rm I}_{2, 1}$
		& $0$
		& $3$ \\
		\hline
		
		\multirow{4}{*}{$\mathcal{P}_{1}\mathcal{P}_{1}^{c}\mathcal{P}_{2}\mathcal{P}_{2}^{c}$}
		& $(\mu_p \times \ZZ / p\ZZ)^4$
		& $4$
		& $0$ \\ \cline{2-4}
		& $(\mu_p \times \ZZ / p\ZZ)^2 \times {\rm I}_{1, 1}^2$
		& $2$
		& $2$ \\ \cline{2-4}
		& $(\mu_p \times \ZZ / p\ZZ) \times {\rm I}_{3, 2}$
		& $1$
		& $2$ \\ \cline{2-4}
		& ${\rm I}_{1, 1}^4$
		& $0$
		& $4$ \\
		\hline
		
		\multirow{3}{*}{$\mathcal{P}_{1}\mathcal{P}_{1}^{c}\mathcal{P}_{2}\mathcal{P}_{3}$}
		& $(\mu_p \times \ZZ / p\ZZ)^2 \times {\rm I}_{1, 1}^2$
		& $2$
		& $2$ \\ \cline{2-4}
		& $(\mu_p \times \ZZ / p\ZZ) \times {\rm I}_{1, 1} \times {\rm I}_{2, 1}$
		& $1$
		& $2$ \\ \cline{2-4}
		& ${\rm I}_{1, 1}^4$
		& $0$
		& $4$ \\
		\hline
		
		$\mathcal{P}_{1}\mathcal{P}_{2}\mathcal{P}_{3}\mathcal{P}_{4}$
		& ${\rm I}_{1, 1}^4$
		& $0$
		& $4$ \\
		\hline
		
		\multirow{3}{*}{$\mathcal{P}_{1}\mathcal{P}_{1}^{c}\mathcal{P}_{2}\mathcal{P}_{2}^{c}\mathcal{P}_{3}$}
		& $(\mu_p \times \ZZ / p\ZZ)^3 \times {\rm I}_{1, 1}$
		& $3$
		& $1$ \\ \cline{2-4}
		& $(\mu_p \times \ZZ / p\ZZ)^2 \times {\rm I}_{2, 1}$
		& $2$
		& $1$ \\ \cline{2-4}
		& $(\mu_p \times \ZZ / p\ZZ) \times {\rm I}_{1, 1}^3$
		& $1$
		& $3$ \\
		\hline
		
		$\mathcal{P}_{1}\mathcal{P}_{1}^{c}\mathcal{P}_{2}\mathcal{P}_{3}\mathcal{P}_{4}$
		& $(\mu_p \times \ZZ / p\ZZ) \times {\rm I}_{1, 1}^3$
		& $1$
		& $3$ \\
		\hline
		
		\multirow{2}{*}{$\mathcal{P}_{1}\mathcal{P}_{1}^{c}\mathcal{P}_{2}\mathcal{P}_{2}^{c}\mathcal{P}_{3}\mathcal{P}_{3}^{c}$}
		& $(\mu_p \times \ZZ / p\ZZ)^4$
		& $4$
		& $0$ \\ \cline{2-4}
		& $(\mu_p \times \ZZ / p\ZZ)^2 \times {\rm I}_{1, 1}^2$
		& $2$
		& $2$ \\
		\hline
		
		$\mathcal{P}_{1}\mathcal{P}_{1}^{c}\mathcal{P}_{2}\mathcal{P}_{2}^{c}\mathcal{P}_{3}\mathcal{P}_{4}$
		& $(\mu_p \times \ZZ / p\ZZ)^2 \times {\rm I}_{1, 1}^2$
		& $2$
		& $2$ \\
		\hline
		
		$\mathcal{P}_{1}\mathcal{P}_{1}^{c}\mathcal{P}_{2}\mathcal{P}_{2}^{c}\mathcal{P}_{3}\mathcal{P}_{3}^{c}\mathcal{P}_{4}$
		& $(\mu_p \times \ZZ / p\ZZ)^3 \times {\rm I}_{1, 1}$
		& $3$
		& $1$ \\
		\hline
		
		$\mathcal{P}_{1}\mathcal{P}_{1}^{c}\mathcal{P}_{2}\mathcal{P}_{2}^{c}\mathcal{P}_{3}\mathcal{P}_{3}^{c}\mathcal{P}_{4}\mathcal{P}_{4}^{c}$
		& $(\mu_p \times \ZZ / p\ZZ)^4$
		& $4$
		& $0$ \\
		\hline
		
		\end{longtable}
		\end{center}
		\normalsize
		
		\end{subsubsection}
		
		\begin{subsubsection}{{\normalsize The new result for $g = 5$}}
		
		\small
		\begin{center}
		\begin{longtable}{| c | c | c | c |}
		\hline
		{\sf ideal decomposition}
		& {\sf group scheme decompotision}
		& {\sf $p$-rank}
		& {\sf $a$-number} \\
		\hline
		
		\multirow{3}{*}{$\mathcal{P}$}
		& ${\rm I}_{5, 1}$
		& $0$
		& $1$ \\ \cline{2-4}
		& ${\rm I}_{5, 3}$,
		  ${\rm J}_{5, 3}$
		& $0$
		& $3$ \\ \cline{2-4}
		& ${\rm I}_{1, 1}^5$
		& $0$
		& $5$ \\
		\hline
		
		\multirow{3}{*}{$\mathcal{P}\mathcal{P}^{c}$}
		& $(\mu_p \times \ZZ / p\ZZ)^5$
		& $5$
		& $0$ \\ \cline{2-4}
		& ${\rm I}_{5, 2}$,
		  ${\rm J}_{5, 2}$
		& $0$
		& $2$ \\ \cline{2-4}
		& ${\rm I}_{5, 4}$
		& $0$
		& $4$ \\
		\hline
		
		\multirow{2}{*}{$\mathcal{P}_{1}\mathcal{P}_{2}$}
		& ${\rm I}_{1, 1} \times {\rm I}_{4, 1}$,
		  ${\rm I}_{2, 1} \times {\rm I}_{3, 1}$
		& $0$
		& $2$ \\ \cline{2-4}
		& ${\rm I}_{1, 1} \times {\rm I}_{4, 3}$,
		  ${\rm I}_{1, 1}^3 \times {\rm I}_{2, 1}$
		& $0$
		& $4$ \\
		\hline
		
		\multirow{8}{*}{$\mathcal{P}_{1}\mathcal{P}_{1}^{c}\mathcal{P}_{2}$}
		& $(\mu_p \times \ZZ / p\ZZ)^4 \times {\rm I}_{1, 1}$
		& $4$
		& $1$ \\ \cline{2-4}
		& $(\mu_p \times \ZZ / p\ZZ)^3 \times {\rm I}_{2, 1}$
		& $3$
		& $1$ \\ \cline{2-4}
		& $(\mu_p \times \ZZ / p\ZZ)^2 \times {\rm I}_{3, 1}$
		& $2$
		& $1$ \\ \cline{2-4}
		& $(\mu_p \times \ZZ / p\ZZ) \times {\rm I}_{4, 1}$
		& $1$
		& $1$ \\ \cline{2-4}
		& $(\mu_p \times \ZZ / p\ZZ) \times {\rm I}_{4, 3}$
		& $1$
		& $3$ \\ \cline{2-4}
		& ${\rm I}_{1, 1} \times {\rm I}_{4, 2}$,
		  ${\rm I}_{3, 2} \times {\rm I}_{2, 1}$,
		& \multirow{2}{*}{$0$}
		& \multirow{2}{*}{$3$} \\
		& ${\rm I}_{1, 1}^2 \times {\rm I}_{3, 1}$,
		  ${\rm I}_{1, 1} \times {\rm I}_{2, 1}^2$
		&
		& \\ \cline{2-4}
		& ${\rm I}_{1, 1}^5$
		& $0$
		& $5$ \\
		\hline
		
		\multirow{2}{*}{$\mathcal{P}_{1}\mathcal{P}_{2}\mathcal{P}_{3}$}
		& ${\rm I}_{1, 1}^2 \times {\rm I}_{3, 1}$,
		  ${\rm I}_{1, 1} \times {\rm I}_{2, 1}^2$
		& $0$
		& $3$ \\ \cline{2-4}
		& ${\rm I}_{1, 1}^5$
		& $0$
		& $5$ \\
		\hline
		
		\multirow{7}{*}{$\mathcal{P}_{1}\mathcal{P}_{1}^{c}\mathcal{P}_{2}\mathcal{P}_{2}^{c}$}
		& $(\mu_p \times \ZZ / p\ZZ)^5$
		& $5$
		& $0$ \\ \cline{2-4}
		& $(\mu_p \times \ZZ / p\ZZ)^3 \times {\rm I}_{1, 1}^2$
		& $3$
		& $2$ \\ \cline{2-4}
		& $(\mu_p \times \ZZ / p\ZZ)^2 \times {\rm I}_{3, 2}$
		& $2$
		& $2$ \\ \cline{2-4}
		& $(\mu_p \times \ZZ / p\ZZ) \times {\rm I}_{4, 2}$,
		& \multirow{2}{*}{$1$}
		& \multirow{2}{*}{$2$} \\
		& $(\mu_p \times \ZZ / p\ZZ) \times {\rm I}_{2, 1}^2$
		&
		& \\ \cline{2-4}
		& $(\mu_p \times \ZZ / p\ZZ) \times {\rm I}_{1, 1}^4$
		& $1$
		& $4$ \\ \cline{2-4}
		& ${\rm I}_{1, 1}^2 \times {\rm I}_{3, 2}$
		& $0$
		& $4$ \\
		\hline
		
		\multirow{6}{*}{$\mathcal{P}_{1}\mathcal{P}_{1}^{c}\mathcal{P}_{2}\mathcal{P}_{3}$}
		& $(\mu_p \times \ZZ / p\ZZ)^3 \times {\rm I}_{1, 1}^2$
		& $3$
		& $2$ \\ \cline{2-4}
		& $(\mu_p \times \ZZ / p\ZZ)^2 \times {\rm I}_{1, 1} \times {\rm I}_{2, 1}$
		& $2$
		& $2$ \\ \cline{2-4}
		& $(\mu_p \times \ZZ / p\ZZ) \times {\rm I}_{1, 1} \times {\rm I}_{3, 1}$,
		& \multirow{2}{*}{$1$}
		& \multirow{2}{*}{$2$} \\
		& $(\mu_p \times \ZZ / p\ZZ) \times {\rm I}_{2, 1}^2$
		&
		& \\ \cline{2-4}
		& $(\mu_p \times \ZZ / p\ZZ) \times {\rm I}_{1, 1}^4$
		& $1$
		& $4$ \\ \cline{2-4}
		& ${\rm I}_{1, 1}^2 \times {\rm I}_{3, 2}$,
		  ${\rm I}_{1, 1}^3 \times {\rm I}_{2, 1}$
		& $0$
		& $4$ \\
		\hline
		
		$\mathcal{P}_{1}\mathcal{P}_{2}\mathcal{P}_{3}\mathcal{P}_{4}$
		& ${\rm I}_{1, 1}^3 \times {\rm I}_{2, 1}$
		& $0$
		& $4$ \\
		\hline
		
		\multirow{7}{*}{$\mathcal{P}_{1}\mathcal{P}_{1}^{c}\mathcal{P}_{2}\mathcal{P}_{2}^{c}\mathcal{P}_{3}$}
		& $(\mu_p \times \ZZ / p\ZZ)^4 \times {\rm I}_{1, 1}$
		& $4$
		& $1$ \\ \cline{2-4}
		& $(\mu_p \times \ZZ / p\ZZ)^3 \times {\rm I}_{2, 1}$
		& $3$
		& $1$ \\ \cline{2-4}
		& $(\mu_p \times \ZZ / p\ZZ)^2 \times {\rm I}_{3, 1}$
		& $2$
		& $1$ \\ \cline{2-4}
		& $(\mu_p \times \ZZ / p\ZZ)^2 \times {\rm I}_{1, 1}^3$
		& $2$
		& $3$ \\ \cline{2-4}
		& $(\mu_p \times \ZZ / p\ZZ) \times {\rm I}_{1, 1} \times {\rm I}_{3, 2}$,
		& \multirow{2}{*}{$1$}
		& \multirow{2}{*}{$3$} \\
		& $(\mu_p \times \ZZ / p\ZZ) \times {\rm I}_{1, 1}^2 \times {\rm I}_{2, 1}$
		&
		& \\ \cline{2-4}
		& ${\rm I}_{1, 1}^5$
		& $0$
		& $5$ \\
		\hline
		
		\multirow{3}{*}{$\mathcal{P}_{1}\mathcal{P}_{1}^{c}\mathcal{P}_{2}\mathcal{P}_{3}\mathcal{P}_{4}$}
		& $(\mu_p \times \ZZ / p\ZZ)^2 \times {\rm I}_{1, 1}^3$
		& $2$
		& $3$ \\ \cline{2-4}
		& $(\mu_p \times \ZZ / p\ZZ) \times {\rm I}_{1, 1}^2 \times {\rm I}_{2, 1}$
		& $1$
		& $3$ \\ \cline{2-4}
		& ${\rm I}_{1, 1}^5$
		& $0$
		& $5$ \\
		\hline
		
		$\mathcal{P}_{1}\mathcal{P}_{2}\mathcal{P}_{3}\mathcal{P}_{4}\mathcal{P}_{5}$
		& ${\rm I}_{1, 1}^5$
		& $0$
		& $5$ \\
		\hline
		
		\multirow{4}{*}{$\mathcal{P}_{1}\mathcal{P}_{1}^{c}\mathcal{P}_{2}\mathcal{P}_{2}^{c}\mathcal{P}_{3}\mathcal{P}_{3}^{c}$}
		& $(\mu_p \times \ZZ / p\ZZ)^5$
		& $5$
		& $0$ \\ \cline{2-4}
		& $(\mu_p \times \ZZ / p\ZZ)^3 \times {\rm I}_{1, 1}^2$
		& $3$
		& $2$ \\ \cline{2-4}
		& $(\mu_p \times \ZZ / p\ZZ)^2 \times {\rm I}_{3, 2}$
		& $2$
		& $2$ \\ \cline{2-4}
		& $(\mu_p \times \ZZ / p\ZZ) \times {\rm I}_{1, 1}^4$
		& $1$
		& $4$ \\
		\hline
		
		\multirow{3}{*}{$\mathcal{P}_{1}\mathcal{P}_{1}^{c}\mathcal{P}_{2}\mathcal{P}_{2}^{c}\mathcal{P}_{3}\mathcal{P}_{4}$}
		& $(\mu_p \times \ZZ / p\ZZ)^3 \times {\rm I}_{1, 1}^2$
		& $3$
		& $2$ \\ \cline{2-4}
		& $(\mu_p \times \ZZ / p\ZZ)^2 \times {\rm I}_{1, 1} \times {\rm I}_{2, 1}$
		& $2$
		& $2$ \\ \cline{2-4}
		& $(\mu_p \times \ZZ / p\ZZ) \times {\rm I}_{1, 1}^4$
		& $1$
		& $4$ \\
		\hline
		
		$\mathcal{P}_{1}\mathcal{P}_{1}^{c}\mathcal{P}_{2}\mathcal{P}_{3}\mathcal{P}_{4}\mathcal{P}_{5}$
		& $(\mu_p \times \ZZ / p\ZZ) \times {\rm I}_{1, 1}^4$
		& $1$
		& $4$ \\
		\hline
		
		\multirow{3}{*}{$\mathcal{P}_{1}\mathcal{P}_{1}^{c}\mathcal{P}_{2}\mathcal{P}_{2}^{c}\mathcal{P}_{3}\mathcal{P}_{3}^{c}\mathcal{P}_{4}$}
		& $(\mu_p \times \ZZ / p\ZZ)^4 \times {\rm I}_{1, 1}$
		& $4$
		& $1$ \\ \cline{2-4}
		& $(\mu_p \times \ZZ / p\ZZ)^3 \times {\rm I}_{2, 1}$
		& $3$
		& $1$ \\ \cline{2-4}
		& $(\mu_p \times \ZZ / p\ZZ)^2 \times {\rm I}_{1, 1}^3$
		& $2$
		& $3$ \\
		\hline
		
		$\mathcal{P}_{1}\mathcal{P}_{1}^{c}\mathcal{P}_{2}\mathcal{P}_{2}^{c}\mathcal{P}_{3}\mathcal{P}_{4}\mathcal{P}_{5}$
		& $(\mu_p \times \ZZ / p\ZZ)^2 \times {\rm I}_{1, 1}^3$
		& $2$
		& $3$ \\
		\hline
		
		\multirow{2}{*}{$\mathcal{P}_{1}\mathcal{P}_{1}^{c}\mathcal{P}_{2}\mathcal{P}_{2}^{c}\mathcal{P}_{3}\mathcal{P}_{3}^{c}\mathcal{P}_{4}\mathcal{P}_{4}^{c}$}
		& $(\mu_p \times \ZZ / p\ZZ)^5$
		& $5$
		& $0$ \\ \cline{2-4}
		& $(\mu_p \times \ZZ / p\ZZ)^3 \times {\rm I}_{1, 1}^2$
		& $3$
		& $2$ \\
		\hline
		
		$\mathcal{P}_{1}\mathcal{P}_{1}^{c}\mathcal{P}_{2}\mathcal{P}_{2}^{c}\mathcal{P}_{3}\mathcal{P}_{3}^{c}\mathcal{P}_{4}\mathcal{P}_{5}$
		& $(\mu_p \times \ZZ / p\ZZ)^3 \times {\rm I}_{1, 1}^2$
		& $3$
		& $2$ \\
		\hline
		
		$\mathcal{P}_{1}\mathcal{P}_{1}^{c}\mathcal{P}_{2}\mathcal{P}_{2}^{c}\mathcal{P}_{3}\mathcal{P}_{3}^{c}\mathcal{P}_{4}\mathcal{P}_{4}^{c}\mathcal{P}_{5}$
		& $(\mu_p \times \ZZ / p\ZZ)^4 \times {\rm I}_{1, 1}$
		& $4$
		& $1$ \\
		\hline
		
		$\mathcal{P}_{1}\mathcal{P}_{1}^{c}\mathcal{P}_{2}\mathcal{P}_{2}^{c}\mathcal{P}_{3}\mathcal{P}_{3}^{c}\mathcal{P}_{4}\mathcal{P}_{4}^{c}\mathcal{P}_{5}\mathcal{P}_{5}^{c}$
		& $(\mu_p \times \ZZ / p\ZZ)^5$
		& $5$
		& $0$ \\
		\hline
		
		\end{longtable}
		\end{center}
		\normalsize
		
		\end{subsubsection}
		
	\end{subsection}

\end{section}

\bibliographystyle{plain}

\begin{thebibliography}{10}
	
	\bibitem{Blake}
	C.~Blake.
	\newblock {\em A Deuring criterion for Abelian varieties.}
	\newblock {Bulletin of the London Mathematical Society}, 46(6):1256--1263, 2014.
	
	\bibitem{Dodson}
	B.~Dodson.
	\newblock {\em The structure of Galois groups of CM-fields.}
	\newblock {Trans. Am. Math. Soc.}, 283:1--32, 1984.
	
	\bibitem{Ekedahl}
	T. Ekedahl.
	\newblock {\em On supersingular curves and Abelian varieties.}
	\newblock {Math. Scand.}, 60:151--178, 1987.
	
	\bibitem{Goren}
	Eyal~Z. Goren.
	\newblock {\em On certain reduction problems concerning Abelian surfaces.}
	\newblock {Manuscr. Math.}, 94(1):33--43, 1997.
	
	\bibitem{Kraft}
	H.~Kraft.
	\newblock {\em Kommutative Algebraische p-Gruppen (mit Anwendungen auf
	  p-divisible Gruppen und abelsche Variet\"aten)}.
	\newblock Sonderforsch. Bereich Bonn. 1975.
	
	\bibitem{Milne}
	J.~Milne.
	\newblock {\em Complex multiplication.}
	\newblock 2006.
	\newblock
	
	\bibitem{Moree}
	P.~Moree.
	\newblock {\em The formal series Witt transform.}
	\newblock {Discrete Mathematics},
	295(1-3):143–-160, 2005.
	
	\bibitem{Oort}
	F.~Oort.
	\newblock {\em Simple $p$-kernels of $p$-divisible groups.}
	\newblock {Adv. Math.}, 198(1):275--310, 2005.
	
	\bibitem{Pink}
	R.~Pink.
	\newblock {\em Finite Group schemes.}
	\newblock 2004/05.
	\newblock http://www.math.ethz.ch/~pink/FiniteGroupSchemes.html.
	
	\bibitem{Pries}
	R.~Pries.
	\newblock {\em A short guide to $p$-torsion of Abelian varieties in characteristic $p$.}
	\newblock 2006/09.
	\newblock http://arxiv.org/abs/math/0609658v1.
	
	\bibitem{Sugiyama}
	K.-I.~Sugiyama.
	\newblock {\em On a generalization of Deuring's results}
	\newblock {Finite Fields and Their Applications}, 26:69--85, 2014.
	
	\bibitem{Zaytsev}
	A.~Zaytsev.
	\newblock {\em Generalization of Deuring reduction theorem.}
	\newblock {Journal of Algebra}, 2012.
	
	\bibitem{Program}
	Source code of the program\footnote{
		Tested in GAP v4.8.6.
	}:
	\newblock {\em https://github.com/asmirnov1005/crtav}.
	
\end{thebibliography}

\def\cprime{$'$}

\end{document}